\documentclass{amsart}

\usepackage{amsfonts}
\usepackage{url}
\usepackage{amscd}
\usepackage{amssymb}	

\usepackage[algosection,lined,boxed,commentsnumbered,ruled,vlined]{algorithm2e}

\usepackage[hidelinks]{hyperref}
\usepackage{hyperref}

\newtheorem{remark}{Remark}[section]

\newtheorem{theorem}{Theorem}[section]
\newtheorem{lemma}[theorem]{Lemma}
\newtheorem{definition}[theorem]{Definition}
\newtheorem{proposition}[theorem]{Proposition}
\numberwithin{equation}{section}

\usepackage{amsfonts}
\usepackage{amscd} 
\usepackage{amssymb}
\usepackage{graphicx}
\usepackage{url}
\newcounter{myalgorithm}
\newtheorem{example}[theorem]{Example}

\begin{document}

\title[Product Subset Problem : Applications to number theory and crypto]{Product Subset Problem : Applications to number theory and cryptography}
\author{K. A. Draziotis}
\author{V. Martidis}
\author{S. Tiganourias}
\address {K. A. Draziotis \\
Department of Informatics\\
Aristotle University of Thessaloniki \\
54124 Thessaloniki, Greece}
\email{drazioti@csd.auth.gr}
\address {Vasilis Martidis\\
Department of Informatics\\
Aristotle University of Thessaloniki \\
54124 Thessaloniki, Greece}
\email{vamartid@csd.auth.gr}
\address {Stratos Tiganourias \\
Department of Informatics\\
Aristotle University of Thessaloniki \\
54124 Thessaloniki, Greece}
\email{etiganou97@gmail.com}
% presented. Separate the keywords with commas.

\keywords{Number Theory, Carmichael Numbers, Product Subset Problem, Public Key Cryptography, Naccache-Stern Knapsack cryptosystem, Birthday attack, Parallel algorithms.}
\subjclass[2010]{11Y16, 11A51, 11T71, 94A60, 68R05}
\maketitle

\begin{abstract}
We consider applications of Subset Product Problem (SPP) in number theory and cryptography.  We obtain a probabilistic algorithm that solves SPP and we analyze it with respect time/space complexity and success probability. In fact we provide an application to the problem of finding Carmichael numbers and an attack to  Naccache-Stern knapsack cryptosystem, where we update previous results.
\end{abstract}
{\let\thefootnote\relax\footnotetext{All the authors contributed equally to this research.}}

\section{Introduction}
In the present paper we study the modular version of subset product problem (MSPP). This problem is defined in a similar way as the subset sum problem \cite{papadopoulou,Freize}.
We consider an application to number theory and cryptography. Furthermore, we shall provide an algorithm for solving MSPP based on birthday paradox attack. Finally we analyze the algorithm with respect to success probability and time/space complexity. Our applications concern the problem of finding Carmichael numbers and as far as the application on cryptography, we update previous results concerning an attack to the Naccache-Stern Knapsack (NSK) public key cryptosystem.  We begin with the following definition.
\begin{definition}[Subset Product Problem]
Given a list of integers $L$ and an integer $c$, find a subset of  $L$ whose product is $c.$
\end{definition}
This problem is (strong) NP-complete using a transformation from Exact Cover by 3-Sets (X3C) problem \cite[p. 224]{Garey}, \cite{yao}. Also, see \cite[Theorem 3.2]{koblitz}, the authors proved that is at least as hard as {\it{Clique}} problem (with respect fixed-parameter tractability). In the present paper we consider the following variant. 
%Also, sometimes this problem is called :  Modular Multiplicative Knapsack Problem.
\begin{definition}[Modular Subset Product Problem : ${\rm M
SPP}_\Lambda$]
Given a positive integer  $\Lambda,$ an integer $c\in {\bf{Z}}_{\Lambda}^*$ and a vector 
$(u_0,u_1,...,u_n)\in ({\bf{Z}}_{\Lambda}^{*})^{n+1},$ find a binary vector $m=(m_0,m_1,...,m_n)$ such that
\begin{equation}\label{equation:basic}
c\equiv \prod_{i=0}^{n}u_i^{m_i}\mod{\Lambda}.
\end{equation}
\end{definition}\ \\
The ${\rm{MSPP}_{\Lambda}}(\mathcal{P},c)$ problem can be defined also as follows : Given a finite set $\mathcal{P}\subset {\bf Z}_{\Lambda}^*$ and a number $c\in{\bf Z}_{\Lambda}^*,$ find a subset $\mathcal{B}$ of $\mathcal{P},$ such that
$$\prod_{x\in {\mathcal{B}}} x \equiv c \mod \Lambda. $$

We can define {\rm{MSPP}} for a general abelian finite group $G$ as following. We write $G$ multiplicative.
\begin{definition}[Modular Subset Product Problem for $G$: ${\rm{MSPP}}_G({\mathcal{P}},c)$]
Given an element $c\in G$ and a vector 
$(u_0,u_1,...,u_n)\in G^{n+1},$ find a binary vector $m=(m_0,...,m_n)$ such that, 
\begin{equation}\label{equation:basic_G}
c = \prod_{i=0}^{n}u_i^{m_i}.
\end{equation}
\end{definition}\ \\
Although in the present work we are interested in $G={\bf{Z}}_Q^*$ where $Q$ is highly composite number (the case of Carmichael numbers) or prime (the case of NSK cryptosystem).

\subsection{Our Contribution} 
First we provide an algorithm for solving product subset problem  based on birthday paradox. This approach is not new, for instance see \cite[Section 2.3]{NC}. Here we use a variant of \cite[Section 3]{original}. We study and implement a parallel version of this algorithm. This result to an improvement of the tables provided in \cite{original}. Further, except the cryptanalysis of NSK cryptosystem, we applied our algorithm to the searching of Carmichael numbers. We used a method of Erd\H{o}s, to the problem of finding Carmichael numbers with many prime divisors. We managed to generate a Carmichael number with $19589$  prime factors\footnote{\url{http://tiny.cc/tm6miz}}. Finally, we provide an abstract version of the algorithm in \cite[Section 3]{original}, to the general product subset problem and further we analyze the algorithm as far as the selection of the parameters (this is provided in Proposition \ref{lemma:new_result}).  \\\\
{\bf Roadmap.} This paper is organized as follows. 
In section \ref{section:birthday_attack} we introduce the attack to MSPP based on birthday paradox. We further provide a detailed analysis. Section \ref{sec:applications} is dedicated to applications. In subsection \ref{sec:nsk} we obtain an application of MSPP to Naccache-Stern Knapsack cryptosystem and in subsection \ref{exp_results} we provide some experimental results.  Subsection \ref{section:carmichael} is dedicated to the problem of finding Carmichael numbers with many prime factors. We provide the necessary bibliography and known results. Finally, the last section contains some concluding remarks.
\section{Birthday Attack to Modular Subset Product Problem}\label{section:birthday_attack}
We call density of ${\rm{MSPP}_G}(\mathcal{P},c)$ the positive real number
$$d = \frac{|\mathcal{P}|}{\log_2{|G|}}.$$
If $G={\bf{Z}}_\Lambda^*,$ then
$$d = \frac{|\mathcal{P}|}{\log_2{|{ {\bf Z}_{\Lambda}^{*}}|}}= \frac{|\mathcal{P}|}{\log_2{\phi(\Lambda)}},$$
where $\phi$ is the Euler totient function. In a ${\rm{MSPP}_G}(\mathcal{P},c)$ having a large density, we expect to have many solutions.

A straightforward attack uses birthday paradox paradigm to  ${\rm MSPP}_\Lambda(\mathcal{P},c)$.  Rewriting equivalence (\ref{equation:basic}) as 
$$ \prod_{i=0}^{\alpha}u_i^{m_i} \equiv c\prod_{i=\alpha+1}^{n}u_i^{-m_i}\mod{\Lambda}, $$
for some integer $\alpha \approx  n/2,$ we construct two subsets of ${\mathbb{Z}}_{\Lambda}$, say $U_1$ and $U_2.$ The first contains elements of the form 
$\prod_{i=0}^{\alpha}u_i^{m_i}\mod{\Lambda},$  and the second 
$ c\prod_{i=\alpha+1}^{n}u_i^{-m_i}{\mod \Lambda},$
for all possible (binary) values of $\{m_i\}_i.$
 So, the problem reduces to finding a common element of sets  $U_1$ and  $U_2.$ 
Below, we provide the pseudocode of the previous algorithm.\\\\
\textbf{ {Algorithm \refstepcounter{myalgorithm}\label{1} \arabic{myalgorithm}: Birthday attack to ${\rm MSPP}_\Lambda\mathcal(\mathcal{P},c)$}}\\
		{\texttt{INPUT\ : $\mathcal{P}=\{u_i\}_i\subset {\bf Z}_{\Lambda}^*$ $(|\mathcal{P}|=n+1),$ $c\in {\bf Z}_{\Lambda}^*$ (assume that $\gcd(u_i,\Lambda)=1$)\\}
	 {\texttt{OUTPUT: $\mathcal{B}\subset \mathcal{P}$ such that  $\prod_{x\in {\mathcal{B}}}x\equiv c \mod{\Lambda}$ or Fail : There is not any solution}}\\\\
	\texttt{1:}   {\texttt{$I_1\leftarrow \{0,1,...,\lceil n/2 \rceil \}, I_2\leftarrow \{\lceil n/2 \rceil+1,...,n\}$}}
	\\
	\texttt{2:} $U_1 \leftarrow \Big\{ \prod_{i\in I_1} u_i^{\varepsilon_i}\mod \Lambda :\text{for \ all}\ \varepsilon_i\in\{0,1\} \Big\}$\\
	\texttt{3:} $U_2 \leftarrow \Big\{ c\prod_{i\in I_2} u_i^{-\varepsilon_i}\mod \Lambda: \text{for \ all}\ \varepsilon_i\in\{0,1\} \Big\}$\\
    \texttt{4:}    {\bf If} $U_1\cap U_2\not= \emptyset$ \\
    \texttt{5:}    \hspace*{0.4cm} Let $y$ be an element of $U_1\cap U_2$\\
	\texttt{6:}    \hspace*{0.4cm} {\texttt{return} $ u_i : \prod u_i \equiv c\mod{\Lambda}.$\\
    \texttt{7:}    {\bf{else}} {\texttt{return}} Fail : There is not any solution\\\\
This algorithm is deterministic, since if there is a solution to  ${\rm MSPP}_\Lambda\mathcal(\mathcal{P},c)$ the algorithm will find it.
To construct the solution in step 6, we use the equation
\begin{equation}\label{birthday_equation}
\prod_{i\in I_1}u_i^{\varepsilon_i} = c \prod_{j\in I_2}u_j^{-\varepsilon_j} \ \ (\rm{in}\ \ {\bf Z}_\Lambda).
\end{equation} 
It turns out $y = \prod_{i=0}^{n} u_i^{\varepsilon_i} \mod{\Lambda}.$
For storage we need $2^{n/2+1}$ elements of ${\bf Z}_\Lambda^{*}.$ In line 4 we compute a common element. To do this we first sort the elements of $U_1, U_2$ and then we apply binary search. Overall we need $O(2^{n/2}\log_2{n})$ arithmetic operations in the multiplicative group ${\bf Z}_\Lambda^{*}.$ The drawback of this algorithm is the large space complexity.

We can improve the previous algorithm as far as the space complexity. %First, we need to define the notion of hamming weight of $c.$
\section{An Improvement of Algorithm \ref{1}}
We provide the following definitions.
\begin{definition}
Let $c \in {\bf{Z}}_{\Lambda}^*$. We define
$${\rm{sol}}(c;{\mathcal{P}},\Lambda) = \{I\subset \{0,1,...,n\}:c\equiv\prod_{i\in I,\ u_i\in\mathcal{P}}u_{i}\mod \Lambda \}.$$  
\end{definition}
%From this set we can easily build the set of all solutions of the ${\rm MSPP}_\Lambda\mathcal(\mathcal{P},c).$ Indeed, if $J_1, J_2,...,J_k$ are all the sets in ${\rm{sol}}(c;{\mathcal{P}},\Lambda),$ then the set of solution is  
%$$\Big\{\{u_i\} : i\in J_1\Big\}\bigcup \Big\{\{u_i\} : i\in J_2\Big\}\bigcup \cdots \bigcup\Big\{\{u_i\} : i\in J_k\Big\}.$$
%We can simplify if we use the following characteristic function 
Let the map,
$$\chi:{\rm{sol}}(c;{\mathcal{P}},\Lambda)\rightarrow \{0,1\}^{n+1},$$
such that $\chi(I)=(\varepsilon_0,\varepsilon_1,...,\varepsilon_n),$ where $\varepsilon_i=1$ if $i\in I$ else $\varepsilon_i=0.$
\begin{definition}
We define,
$${\rm{Sol}}(c;{\mathcal{P}},\Lambda) = \{(u_0^{\varepsilon_0},u_1^{\varepsilon_1},...,u_n^{\varepsilon_n}): \chi(I) = (\varepsilon_0,\varepsilon_1,...,\varepsilon_n)\ \text{for\ all}\ I\in {\rm{sol}}(c;{\mathcal{P}},\Lambda)\}.$$
\end{definition}
\begin{definition}
To each element $I$ of ${\rm{sol}}(c;\mathcal{P},\Lambda)$ (assuming there exists one) we correspond the natural number $H_{I}(c)=|I|$ (the cardinality of $I$). We call this number local Hamming weight of $c$ at $I.$ We call Hamming weight   $H(c)$ of $c,$ the minimum of all these numbers. I.e.
$$H(c) = \min\{H_I(c): I\in  {\rm{Sol}}(c;{\mathcal{P}},\Lambda)\}.$$
\end{definition}
Remark that the local Hamming weight of $c$ at $I$ is in fact the Hamming weight of the binary vector $\chi(I).$ That is, the sum of all the entries of $\chi(I).$ Thus,
$$H_I(c) = \sum_{i=0}^{n} \varepsilon_i, {\rm{where}} \ \chi(I)=(\varepsilon_0,\varepsilon_1,...,\varepsilon_n).$$
\begin{example}
 Let $\mathcal{P}=\{2,3,...,10^{7}\},$ $\Lambda=10000019,$ $c=190238.$ 
 There is an element $\Sigma$ of ${\rm{Sol}}(c;{\mathcal{P}},\Lambda),$ 
   $$\Sigma = (9851537, 303860, 4680021, 9647209, 2006838, 9984877,$$ $$2512434, 2126904, 1942182, 8985302, 2193757).$$
   Note that we have trimmed all the ones. I.e. 
   $$\Sigma'=(1,..,1,303860,1,...,1942182,1,...,9984877,1...,1).$$
Straightforward calculations provide,
$$\prod_{x\in \Sigma} x = \prod_{x\in \Sigma'} x\equiv 190238=c\pmod{\Lambda}.$$
Also there is an element in ${\rm{sol}}(c;{\mathcal{P}},\Lambda)$ say $I,$ such that,
$$\chi(I)=\Sigma'.$$ 

So in this case the local hamming weight of $c$ at $I$ is $11.$ However, the computation of 	the hamming weight of $c$ is more difficult. Note that $H(c)\geq 2,$ except if some $u_i$ is equal to $c.$ We do not consider this trivial case. For now, we can only say that $H(c)\leq 11.$ I.e. every local Hamming weight is an upper bound for $H(c).$
\end{example}
Now, let $I\in {\rm{sol}}(c;{\mathcal{P}},\Lambda).$ We consider two positive integers, say $h_1,h_2,$ such that, $h_1+h_2=H_I(c),$ 
%We consider also the set $\mathcal{P}=\{u_0,...,u_n\}$ 
and two disjoint subsets $I_1,\ I_2$ of $\{0,1,\dots,n\}$ with $|I_1|=|I_2|=b,$ for some positive integer $b \leq n/2.$
Finally, we consider the sets,
$$U_{h_1}(I_1;\mathcal{P},\Lambda) = \Big\{\prod_{i\in I_1} u_i^{\varepsilon_i}\pmod \Lambda:\sum_{i\in I_1} \varepsilon_i = h_1\Big\},$$
$$U_{h_2}(I_2,c;\mathcal{P},\Lambda) = \Big\{c\prod_{i\in I_2} u_i^{-\varepsilon_i}\pmod \Lambda:\sum_{i\in I_2} \varepsilon_i = h_2\Big\}.$$
%The set $U_1$ of  algorithm 1 is the $U_{h_1}$ for $h_1=\lceil n/2  \rceil.$
We usually write them as $U_{h_1}(I_1)$ and $U_{h_2}(I_2,c)$ since $\mathcal{P}$ and $\Lambda$ are known.
We have,
$$|U_{h_1}(I_1)|=\binom{|I_1|}{h_1},\ |U_{h_2}(I_2,c)|=\binom{|I_2|}{h_2}.$$ 
%In our algorithm $|I_1|$ and $|I_2|$ are the same (or very close) and also are disjoint subsets of $\{0,1,\dots,n\}.$
\begin{remark}
The set $U_1$ of Algorithm \ref{1} is written,
$$U_1 = \bigcup_{h_1=1}^{\lceil n/2\rceil +1}U_{h_1}(\{0,1,\dots,\lceil n/2 \rceil\};\mathcal{P},\Lambda).$$
I.e. $U_1$ is the union of the sets $\{U_{h}(I;\mathcal{P},\Lambda)\}_{1\leq h\leq \lceil n/2\rceil +1}$ for $I=\{0,1,...,\lceil n/2\rceil\}.$
Similar $U_2$ is written,
$$U_2 = \bigcup_{h_2=0}^{n - \lceil \frac{n}{2}\rceil}U_{h_2}(\{\lceil n/2 \rceil+1,...,n\},c;\mathcal{P},\Lambda).$$
\end{remark}

Instead of using $U_1,U_2$ we use subsets of them. The choice of subsets creates a probability distribution at the output. So this choice must be studied as fas as the success probability. 
For the following algorithm we assume that we know a local Hamming weight of the target number $c.$\ \\\\
{\bf Algorithm\refstepcounter{myalgorithm}\label{2} \arabic{myalgorithm} ${\rm BA\_MSPP}_{\Lambda}({\mathcal{P}},c;b,\ell,Q,{\rm iter}) $\footnote{BA : Birthday Attack} : Memory efficient attack to ${\rm MSPP}_\Lambda(\mathcal{P},c)$}\\
	{\texttt{INPUT: \\ ${\bf i}.$ A set $\mathcal{P}=\{u_i\}_i\subset {\bf Z}_{\Lambda}^*$ with $|\mathcal{P}|=n+1$ (assume that $\gcd(u_i,\Lambda)=1$)\\
	${\bf ii}.$ a number $c\in {\bf Z}_{\Lambda}^*$\\ 
	${\bf iii}.$  a local Hamming weight of $c,$ say $\ell$\\
	 ${\bf iv}.$  a positive number $b :$ $\ell \leq b\leq n/2$\\
%	 ${\bf v}.$  a hash function ${\mathbb{H}},$ where we assume that the output is given as a series of hex characters.
${\bf v}.$  a compression function ${\mathbb{H}}$
 		\\ 
% 		${\bf vi}.$  $Q:$ is a positive integer. ${\mathbb{H}}_Q(x)$ is the sequence consists from the first $4Q$ bits (or $Q-$hex characters) of the hash function ${\mathbb{H}}$ applied to some number $x$.\\
 		${\bf vi}.$  a positive integer {\it{iter}}\\\\}
	 {\texttt{OUTPUT: a set $\mathcal{B}\subset \mathcal{P},$ such that $\prod_{x\in {\mathcal{B}}}x\equiv c \mod{\Lambda}$ or Fail}}\\\\
	 \texttt{1:} $(h_1,h_2)\leftarrow (\lfloor \ell/2\rfloor,\lceil \ell/2\rceil)$\\
	 \texttt{2:} {\bf For}} $i$ in $1,\dots ,iter$ \\
	 \texttt{3:}  \hspace*{0.3cm} \  {{$(I_1,I_2) \xleftarrow{\$}\{0,...,n\}\times \{0,...,n\}$ \ \# {\small with $I_1,I_2$ disjoint and $|I_1|=|I_2|=b$ 
	 }}}\\
	 \texttt{4:} \hspace*{0.4cm} $U_{h_1}^*(I_1) \leftarrow \Big\{ {\mathbb{H}}\big{(}\prod_{i\in I_1} u_i^{\varepsilon_i}\pmod \Lambda \big{)}:\sum_{i\in I_1}\varepsilon_i=h_1,\ \varepsilon_i\in\{0,1\} \Big\}$\\   
	 \texttt{5:} \hspace*{0.4cm} {\bf For} each $(\varepsilon_i)_i$ such that $:\sum_{i\in I_2}\varepsilon_i=h_2$\\
	 \texttt{6:} \hspace*{0.8cm} {\bf If\ }  ${\mathbb{H}}\big{(}c\prod_{i\in I_2} u_i^{-\varepsilon_i}\pmod \Lambda \big{)}\in U_{h_1}^*(I_1)$\\   
    \texttt{7:}    \hspace*{1.2cm} Let $y$ be an element of $U_{h_1}^*(I_1)\cap U_{h_2}^*(I_2,c)$\\
	\texttt{8:}    \hspace*{1.2cm} {\texttt{return} $\{ u_i\}_i $ such that $\prod u_i \equiv c\mod{\Lambda}$ \texttt{and terminate}\\
    \texttt{9:}   {\texttt{return} Fail}\ \# {\small if for all the iterations the algorithm failed to find a solution}\\\\
    
    This algorithm is a memory efficient version of algorithm \ref{1}, since we consider subsets of $U_1, U_2.$ Although, this algorithm may fail, even when ${\rm MSPP}_\Lambda(\mathcal{P},c)$ has a solution. For instance, assuming that ${\rm{sol}}(c;\mathcal{P},\Lambda)\not=\emptyset,$ if we pick $I_1, I_2$ and happens that the union $I_1\cup I_2 \not\in {\rm{sol}}(c;\mathcal{P},\Lambda),$ then the algorithm will fail. I.e. the problem may have a solution but the algorithm failed to find it. This may occur when $b<n/2,$ that is $I_1\cup I_2 \subset \{0,1,...,n\}.$ 
    If  $I_1\cup I_2 = \{0,1,...,n\},$  then the algorithm remains probabilistic, since we consider a specific choice of $(h_1,h_2)$ and not all the possible $(h_1,h_2),$ with $h_1+h_2=\ell.$ If we consider all $(h_1,h_2)$ such that $h_1+h_2=n,$ then the algorithm turns out to be deterministic. %In this case the analysis below is similar.

We analyze the algorithm line by line.\\
{\it Line 3:} This can be implemented easily in the case where $2b<\sqrt{n}.$ Indeed, we can use rejection sampling in the set $\{0,...,n\}$ and construct a list of length $2b.$ Then, $I_1$ is the set consisting from the first $b$ elements and $I_2$ the rest. If $2b\geq \sqrt{n}$ then, we have to sample from the set $\{0,1,...,n\},$ so the memory increases since we have to store the set $\{0,1,...,n\}.$ For instance, when we apply the algorithm to the searching of Carmichael numbers we have $2b\ll \sqrt{n}.$ In the case of NSK cryptosystem we usually have $2b>\sqrt{n}.$ \\
{\it Line 4:} The most intensive part (both for memory and time complexity) is the construction of the set $U_{h_1}^*$. Here we can parallelize our algorithm to decrease time complexity. To reduce the space complexity we use the parameter $b\le n/2$ and the compression function ${\mathbb{H}}.$ For instance as a a compression function we  consider $Q-$ strings of the output of a hash function. We consider that the output of the hash function is a hex string. In subsection \ref{how_to_choose_Q} we provide a strategy to choose $Q.$\\
{\it Line 5-6:} In Line 4 we stored $U_{h_1}^*(I_1),$ in this line  we compute {\it{on the fly}} the elements of the second set $U_{h_2}^*(I_2,c)$ and check if any is in $U_{h_1}^*(I_1).$ So we do not need to store $U_{h_2}^*.$ A suitable data structure for the searching is the hashtable, which we also used in our implementation.  Hashtables have the advantage of the fast insert, delete and search operations. Since these operations have $O(1)$ time complexity on the average.\\
{\it Line 7-8:} Having the element found by the previous step  (Line 6), say $y,$ we construct the $u_i$'s such that their product is $y.$ We return Fail if the intersection is empty for all the iterations.
\begin{remark}
If we do not consider any hash function and ${\rm iter}=1,$ then we write 
${\rm BA\_MSPP}_{\Lambda}({\mathcal{P}},c;b,\ell).$
\end{remark}

\subsection{Space complexity}

We assume that there is not any collision in the construction 
of $U_{h_1}^*(I_1)$ and $U_{h_2}^*(I_2,c),$ or in other words we choose $Q$ to minimize the probability to have a collision. I.e. we choose $2^{4Q}\gg \max\{|U_1|,|U_2|\}$ and we assume that ${\mathbb{H}}$ is behaving random enough. In practice (or at least in our examples) we always have this constraint.

So we get, 
$$|U_{h_1}(I_1)|=|U_{h_1}^*(I_1)|=\binom{|I_1|}{h_1}\ \text{ and}\ \ |U_{h_2}(I_2,c)|=|U_{h_2}^*(I_2,c)|=\binom{|I_2|}{h_2},$$ 
where $(h_1,h_2) = (\lfloor \ell /2\rfloor,\lceil \ell/2\rceil).$ By choosing $|I_1|=|I_2|= b,$ we get
 $$|U_{h_1}^*(I_1)|= \binom{b}{h_1}\ \ \text{and}\ \ |U_{h_2}^*(I_2,c)|= \binom{b}{h_2}.$$
 We set 
 $$S_b = \binom{b}{h_1} = B_{b}(h_1).$$
 In our algorithm, we need to store the $S_b$ hashes  of the set $U_{h_1}^*(I_1).$ We need $4Q-$bits for keeping $Q-$hex digits in the memory. So, overall we store $4QS_b$ bits. Furthermore, we store   the binary exponents $(\varepsilon_i)_i$ that are necessary for the computation of the products, $\prod_{i\in I_1} u_i^{\varepsilon_i}.$ This is needed, since we must reconstruct $c$ as a product of $(u_i)_i.$ These are $S_b,$ thus we need $b\times S_b-$ bits. 
 Since we also need to store the set $\mathcal{P}$ of length $n+1,$ we conclude that,
 \begin{equation}\label{formula:memory}
 \mathbb{M} < (4Q + b)S_b + (n+1)B\ \ (\text{bits}),
 \end{equation}
where $B = \max\{\log_{2}(x):x\in \mathcal{P}\}.$
Remark that ${\mathbb{M}}$ does not depend on the modulus $\Lambda.$ 
%but only from the local Hamming weight, the bound $b$ and the parameter $Q.$

\begin{table}[h]
\begin{tabular}{|l||l|l|l|l|l|l|l|}
\hline 
$\ell$ & 9 & 11 & 13 & 15 & 17 & 19 & 21\\ \hline
${\mathbb{M}}\ (GB)$ &  0.029 & 0.05 & 0.2 & 1.16 & 6.15 & 28.61 & 117.2 \\ \hline
\end{tabular}
\caption{For $b=50,\ Q=12$ and $\mathcal{P}=\{2,3,...,10^7\}$.}
\label{Tab:1}
\end{table} 
If we can describe in an efficient way the set 
$\mathcal{P}$ we do not need to store it. 
Say, that $\mathcal{P}=\{2,3,...,n\}.$ Then, there is no need to 
store it in the memory, since the sequence $f(x) = x+1$ 
describes efficiently the set $\mathcal{P}.$ Also, in other situations $B$ can be stored using $O(|{\mathcal{P}}|\log_{2}(|{\mathcal{P}}|))$ bits. We call such sets {\it nice} and they can save us enough memory. In fact, for nice sets the inequality (\ref{formula:memory}) changes to,
\begin{equation}\label{formula:memory_for_nice_sets}
 \mathbb{M} < (4Q + b)S_b + O(n\log_{2}n)\ \ (\text{bits}).
 \end{equation}
 In fact when we apply this algorithm to the problem of finding Carmichael numbers, we shall see that the set $\mathcal{P}$ is {\it nice}.
 
 Finally, if $U_{h_1}$ is very large we can make chunks of it, to store it in the memory. Note that this can not be done if we directly compute the intersection of $U_{h_1}\cap U_{h_2}$ as in \cite{original}. This simple trick considerable improves the algorithms in \cite{original}.
\subsection{Time complexity}
Time complexity is dominated by the construction of the sets $U_{h_1}$ and $U_{h_2}$ and the calculation of their intersection.
We work with $U_{h_1}$ instead of $U_{h_1}^*$ since all the multiplications are between the elements of $U_{h_1}$, only in the searching phase we move to $U_{h_1}^*$.
Let $M_\Lambda$ be the bit-complexity of the multiplication of two integers $\mod \Lambda.$ So 
$$M_\Lambda=O( (\log_2{\Lambda})^{1+\varepsilon})$$ for some $0<\varepsilon\leq 1$ (for instance Karatsuba suggests $\varepsilon =\log_{2}{3}-1$ \cite{karatsuba}). In fact recently was proved $M_\Lambda= O(\log_2{\Lambda}\log_2{(\log_2{\Lambda}}))$ \cite{nlogn}.
To construct the sets $U_{h_1}$ and $U_{h_2}$ (ignoring the cost for the inversion ${\mod \Lambda}$)  we need $M_\Lambda h_1B_{b}(h_1)=M_\Lambda h_1 S_b$ bit-operations} for the set $U_{h_1}$ and  $M_\Lambda h_2B_b(h_2+1)\approx M_\Lambda h_2S_b$ bit-operations for $U_{h_2}.$
So overall,
$${\mathbb{T}}_1 = M_\Lambda (h_1S_b+h_2S_b)=M_\Lambda S_b(h_1+h_2) \ \text{bits}.$$
Considering the time complexity for finding a collision in the two sets by using a hashtable, we get 
$${\mathbb{T}} = {\mathbb{T}}_1  + O(1)S_b\ \text{bits}\ \text{on\ average}$$
and 
$${\mathbb{T}} = {\mathbb{T}}_1  + O(1)S_b^2\ \text{in\ the\ worst\ case}.$$
We used that $h_1\approx h_2$ (they differ at most by $1$). 
In case we have $T$ threads we get about ${\mathbb{T}}/T$ (bit operations) instead of ${\mathbb{T}}.$
\subsection{Success Probability}
In the following Lemma we compute the probability to get a common element in $U_{h_1}(I_1)$ and $U_{h_2}(I_2,c)$ when 
$(I_1,I_2) \xleftarrow{\$}\{0,...,n\}\times \{0,...,n\},$ where $I_1,I_2$ are disjoint, with $b$ elements. Let $0\leq y\leq x.$ With $B_x(y)$ we denote the binomial coefficient $\binom{x}{y}.$
\begin{lemma}\label{lemma:step7}
Let $h_1,h_2$ be positive integers and $\ell = h_1+h_2.$ 
The probability to get $U_{h_1}(I_1)\cap U_{h_2}(I_2,c)\not=\emptyset$ is,
$$
{\mathbb{P}}=\frac{B_b(h_1)B_b(h_2)}{B_{n+1}(\ell)}.
$$
\end{lemma}
For the proof see \cite[section 3]{original}.

We can easily provide another and simpler proof in the case $2b=n+1$ (this occurs very often when attacking Naccache-Stern cryptosystem). Then,  
$${\mathbb{P}} = hyper(x;2b,b,\ell),$$
where {\it hyper} is the hypergeometric distribution,
$$hyper(x;N,b,\ell)=Pr(X=x)=\frac{\binom{b}{x}\binom{N-b}{\ell-x}}{\binom{N}{\ell}}.$$
Where, \\
$\bullet$ $N=n+1$ is the population size\\
$\bullet$ $\ell$ is the number of draws\\
$\bullet$ $b$ is the number of successes in the population\\
$\bullet$ $x$ is the number of observed successes\\
Adapting to our case, we set
$N=2b=n+1, \ell=h_1+h_2,$ and $x=h_1.$ Then,
$$hyper(x=h_1;n+1,b,\ell)=Pr(X=h_1)=\frac{\binom{b}{h_1}\binom{b}{\ell-h_1}}{\binom{n+1}{\ell}}=\mathbb{P}.$$
The expected value is $\frac{b\ell}{n+1}$ and since $b=(n+1)/2$ we get $EX = \frac{\ell}{2}.$ Since the random variable $X$ counts the successes we expect on average to have $\ell/2$ after considering enough instances (i.e. choices of $I_1,$ $I_2$). The maximum value of $\ell$ is $n/2.$ Thus, in this case the expected value is maximized, hence we expect our algorithm to find faster a solution from another one that uses smaller value for $\ell.$

 Also, we need about 
$$\frac{1}{{\mathbb{P}}}=\frac{\binom{n+1}\ell}{\binom{b}{h_1}\binom{b}{h_2}}\  \text{iterations on average to find a solution}.$$

One last remark is that the contribution of $b$ is bigger than the contribution of hamming weight in the probability ${\mathbb{P}}.$

\begin{figure}[!htb]
%\raggedright
\includegraphics[scale=0.34]{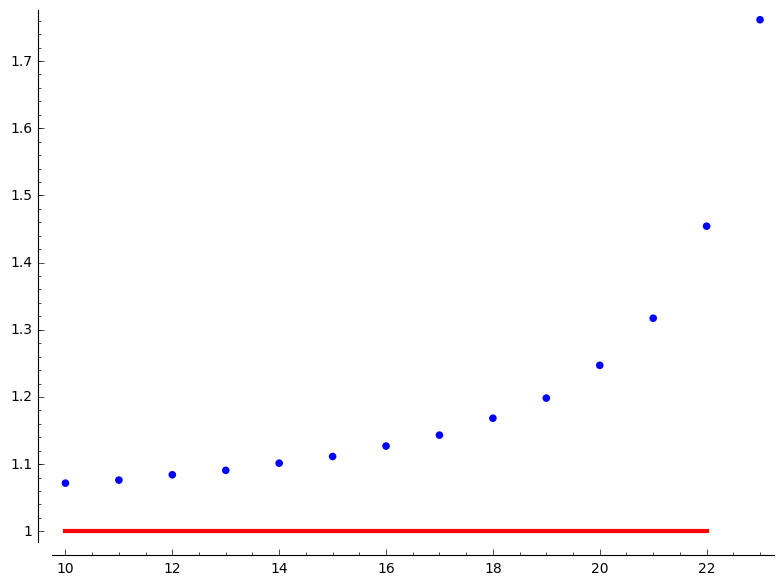}
\caption{We fixed $n = 23000$ and $b = 35.$ We consider pairs $(h,f(h)),$ where $f(h) = \big({\mathbb{P}(b+1,h)}/{\mathbb{P}(b,h+1)}\big{)}$ and $h$ the hamming weight. We finally normalized the values by taking logarithms. So, for a given $(b,h),$ if we want to increase the probability is better strategy to increase $b$ than $h.$}\label{fig:3}
\end{figure}

\subsubsection{The best choice of $h_1$ and $h_2$}\label{section:h1h2}
The choice of $h_1, h_2$ (in line 1 of algorithm \ref{2}) is $h_1 = \lfloor \ell/2\rfloor,$ $h_2 = \lceil \ell/2\rceil.$ This can be explained easily, since these values maximize the probability ${\mathbb{P}}$ of Lemma \ref{lemma:step7}. 
We set 
$$J_{\ell}=\{(x,y)\in {\bf{Z}}^2: x+y=\ell, 0\leq x \leq y\}.$$
Observe that $(\lfloor \ell/2\rfloor,\lceil \ell/2\rceil) \in J_{\ell}.$
\begin{proposition}\label{lemma:new_result}
Let $b$ and $n$ be fixed positive integers such that $\ell=x+y\leq b\leq n/2.$ Then the finite sequence ${\mathbb{P}}:J_{\ell}\rightarrow {\mathbb{Q}},$ defined by
$${\mathbb{P}}(x,y) = \frac{B_b(x)B_b(y)}{B_{n+1}(\ell)},$$ is maximized for $(x,y) = (\lfloor \ell/2\rfloor,\lceil \ell/2\rceil).$ 
\end{proposition}
We need the following simple lemma.
\begin{lemma}
$$\binom{b}{x}\binom{b}{\ell-x}\binom{2b}{b} = \binom{\ell}{x}\binom{2b-\ell}{b-x}\binom{2b}{\ell}.$$
\end{lemma}
\begin{proof}
It is straightforward by expressing the binomial coefficients in terms of factorials and rearranging them. 
\end{proof}

\begin{proof}[Proof of Proposition \ref{lemma:new_result}]
From the previous lemma we have
\begin{equation}\label{simple_binomial_identity}
\frac{\binom{b}{x} \binom{b}{\ell-x} }{\binom{2b}{\ell} } = \frac{\binom{\ell}{x}\binom{2b-\ell}{b-x}}{\binom{2b}{b}}.
\end{equation}

Since $b,\ell$ are fixed, the maximum of the right hand side is at $x = \lfloor \ell/2 \rfloor$. Indeed, both sequences $\binom{b}{x}$ and $\binom{2b-\ell}{b-x}$ are positive and maximized at $x = \lfloor \ell/2 \rfloor.$ So also the product $\binom{\ell}{x}\binom{2b-\ell}{b-x}$ is maximized at $x = \lfloor \ell/2 \rfloor.$
The same occurs to the left hand side. Since the denominator of the left hand side in (\ref{simple_binomial_identity}) is fixed, the numerator $B_b(x)B_b(\ell-x)=\binom{b}{x} \binom{b}{\ell-x}$ is maximized at $x = \lfloor \ell/2 \rfloor.$ Since $n$ is also a fixed positive integer, the numerator of ${\mathbb{P}}$  is maximized at $x = \lfloor \ell/2 \rfloor$ and so ${\mathbb{P}}$ is maximized at the same point. Finally, $\ell - \lfloor \ell/2 \rfloor = \lceil \ell/2 \rceil = y.$ The Proposition follows.
\end{proof}
\subsection{How to choose $Q?$}\label{how_to_choose_Q}
If we are searching for $r-$same objects to one set (with cardinality $n$), when we pick the elements of the set from some largest set (with cardinality $m$), then we say that we have a $r-$multicollision. We have the following Lemma.
\begin{lemma}\label{lemma:multi}
If we have a set with $m$ elements and we pick randomly (and independently) $n$ elements from the set, then the expected number of $r-$mullticolisions is approximately
\begin{equation}\label{multi}
\frac{n^{r}}{r!m^{r-1}}.
\end{equation}
\end{lemma}
\begin{proof}
\cite[section 6.2.1]{Joux}
\end{proof}
Say we use md5 hash function. If we use the parameter $Q$, we have to truncate the output of md5, which has $16-$hex strings, to $Q-$hex strings ($Q<16$).  I.e we only consider the first $\kappa = 4\cdot Q$-bits of the output. 
Our strategy to choose $Q$ uses formula (\ref{multi}). In practice, is enough to avoid $r=3-$multicollisions in the set $U_{h_1}^*(I_1)\cup U_{h_2}^*(I_2,c)$ of cardinality $S_b,$ where 
$$S_b=|U_{h_1}^*(I_1)\cup U_{h_2}^*(I_2,c)| = \binom{|I_1|}{h_1} + \binom{|I_1|}{h_2}.$$ 
We set the formula  (\ref{multi}) equal to 1 and we solve with respect to $m$, which in our case is $m=2^{\kappa}.$ So,
$\frac{S_b^{3}}{6}=2^{2\kappa}.$ Therefore we get,
$\kappa \approx (3\log_2{S_b})/2$ (since in our examples $\log_{2}(S_b)\gg \log_2{6}).$ For instance if we have local Hamming weight $\leq 13,$ $b=n/2,$ $|I_1|=|I_2|=b,$ and $n=232,$ we get $\kappa \approx 52.$ So, $Q\approx 13.$ In fact we used $Q=12$ in our attack to Naccache-Stern knapsack cryptosystem.
\section{Applications}\label{sec:applications}

\subsection{Naccache-Stern Knapsack Cryptosystem}\label{sec:nsk}

In this section we consider a second application of MSPP to cryptography. We shall provide an attack to a public key cryptosystem.  Naccache-Stern Knapsack (NSK) cryptosystem is a public key cryptosystem (\cite{NC}) based on the Discrete Logarithm Problem (DLP), which is a difficult number theory problem. Furthermore, it is based on another combinatorial problem, the Modular Subset Product Problem. Our attack applies to the latter problem. 
NSK cryptosystem is defined by the following three algorithms.\\\\ 
{\bf i}. \textbf{Key Generation:} \\
Let $p$ be a large safe prime number (that is $(p-1)/2$ 	is a prime number). Let  $n$ denotes the largest positive integer such that:
\begin{equation}\label{pandn}
p>\prod_{i=0}^{n}p_{i},
\end{equation}
where $p_{i}$ is the $(i+1)-$th prime.
The message space of the system is 
$\mathcal{M}=\{0,1\}^{n+1},$ %\{2^{n},...,2^{n+1}-1\},$$ 
this is the set of the binary strings of $(n+1)-$bits. For instance, if $p$ has $2048$ bits, then $n=232$ and if $p$ has $1024$ bits, then $n=130.$

We randomly pick a positive integer $s<p-1$, such that $
\gcd(s,p-1)=1$. This last property guarantees that there 
exists the (unique) $s-$th root $\mod{p}$ of an element in 
${\bf Z}_p^{*}.$ Set 
$$u_{i}=\sqrt[s]{p_{i}}  \mod{p}\in {\bf{Z}}_p^{*}.$$
The public key is the vector 
$$(p,n;u_{0},...,u_{n})\in {\bf{Z}}^2 \times ({\bf{Z}}_p^{*})^{n+1} $$ and the secret key is $s$.\\\\
{\bf ii}. \textbf{Encryption:}\\
Let $m$ be a message and $\sum_{i=0}^{n}2^{i}m_{i}$ its binary expansion.
The encryption of the $n+1$ bit message $m$ is $c=\prod_{i=0}^{n}u_{i}^{m_{i}}  \mod{p}$.\\
{\bf iii}. \textbf{Decryption:}\\
To decrypt the ciphertext $c,$  we compute 
$$m=\sum_{i=0}^{n}\frac{2^{i}}{p_{i}-1}\times \big(\gcd(p_{i},c^{s}\mod{p})-1\big).$$  

From the description of the NSK scheme, we see that the security is based on the Discrete Logarithm Problem (DLP). It is sufficient to solve $u_i^x=p_i$ in ${\bf Z}_p^{*},$ for some $i.$ The best algorithm for computing DLP in prime fields has subexponential bit complexity, \cite{dlp-sub1,dlp-sub2}. Thus, for large $p$ (at least $2048$ bits) the system can not be attacked by using the state of the art algorithms for DLP. 

We have also assumed that the prime number $p$ belongs to the special class of safe primes to prevent attacks such as, Pollard rho \cite{RHO}, Pollard $p-1$ algorithm \cite{p-1}, Pohlig-Hellman algorithm \cite{Pohlig} or any similar procedure that exploits properties of $p-1.$ 

\subsection{The attack}
Since, $c\equiv\sum_{i=0}^{n}2^{i}m_i \mod{\Lambda},$
we get
\begin{equation}\label{birthday_equation}
\prod_{i\in I_1}u_i^{m_i} = c \prod_{i\in I_2}u_i^{-m_i} \ \ (\rm{in}\ \ {\bf Z}_\Lambda).
\end{equation} 
So we can apply ${\rm BA\_MSPP}_{\Lambda}$ with input ${\mathcal{P}}=(u_i)_i$ and $c$ and for some bound $b$ and hamming weight of the message $m$ say $H_m$ i.e. the number of $1'$s in the binary message $m.$ So in this attack we assume that we know the hamming weight or an upper bound of it. To be more precise, this attack is feasible only for small or large hamming weights. Our parallel version allow us to consider larger hamming weights than in \cite{original}. 

In the following algorithm we call algorithm \ref{2}, where we execute steps 4 and 5 in parallel (in function ${\rm BA\_MSPP}_{\Lambda}$).
\ \\	\\
\textbf{ {Algorithm\refstepcounter{myalgorithm}\label{3} \arabic{myalgorithm} : Attack to NSK cryptosystem}}\\
		{\texttt{INPUT: $\circ$ The cryptographic message $c$ 
		\\$\circ$  the Hamming weight $H_m$ of the message $m$ 
		\\  $\circ$ a bound $b\leq n/2$ 
		\\ $\circ$  the public key $pk = (p,n;u_0,...,u_n)$ of NSK cryptosystem}\\\\
	 {\texttt{OUTPUT: the message $m$ or Fail}\\\\
	\texttt{1:}   ${\mathcal{P}}\leftarrow \{u_0,...,u_n\}$\\
	\texttt{2:}   ${\mathbb{S}}\leftarrow {\rm{BA\_MSPP}}_{p}({\mathcal{P},c;b,H_m})$\\
	\texttt{3:}   if ${\mathbb{S}}\not= \emptyset$ construct $m.$ Else return {\texttt{Fail}}

\subsubsection{Reduction of the case of large Hamming weight messages}

The case where we have large Hamming weight of a message can be reduced to the case where we have small Hamming weight.
Indeed, if the message $m$ has $H_m=n+1-\varepsilon,$ where $\varepsilon$ is a small positive integer, then again we can reduce the problem to one with small Hamming weight. Let $c=Enc(m)$ and 
$c'=c^{-1}u_n^2\prod_{i=0}^{n-1}u_i.$ We provide the following Lemma.
\begin{lemma}\label{lemma}
The decryption of $c'$ is $m'=2^{n+1}+2^{n}-m-1,$ where $H_{m'}=\varepsilon + 1.$
\end{lemma}
\begin{proof}
\cite[Lemma 3.4]{original}
\end{proof}	

So we can consider $H_m<b.$ Indeed, if $H_m\geq b,$ we apply the previous Lemma and we get $H_{m'} < b.$ So finding $m'$ is equivalent finding $m.$ 
\subsubsection{The case of knowing some bits of the message}
If we know the position of some bits of the message $m$ (for instance by applying a fault attack to the system may leak some bits), then we can improve our attack. In this case, we choose $I_1$ and $I_2$ in algorithm \ref{3}, 
from the set $\{0,1,..,n\}-K,$ where the set $K$ contains the positions of the known bits. Also, in line 10, when we reconstruct the message $m$ (in case of a collision) we need to put the known bits to the right positions.

\subsection{Experimental results for NSK cryptosystem}\label{exp_results}
In our implementation\footnote{The code can be found in \url{https://goo.gl/t9Fa68}} we used C/C++ with GMP library \cite{gmp} and for parallelization OpenMP \cite{openmp}. We used $20$ threads of an Intel(R) Xeon(R) CPU E5-2630 v4 $@ 2.20GHz,$ in a Linux platform.

First, in table \ref{Tab:results-1} we present the improvement of the results provided in \cite{original} by using the parallel version.
Besides, in table \ref{Tab:results-2} we extend the results of the previous table. In fact, table \ref{Tab:results-2} demonstrates that having a suitable number of threads and considering a suitable bound $b$ we get a practical attack for low Hamming weight messages. 
In figure \ref{fig:1} we represent some of our data graphically.
\begin{table}[h!]
\begin{tabular}{|l||lll||lll||ll|}
\hline 
${\rm len}(p)$  & & \ \  $600$ &  &  &$1024$ & &  \ \ \ \ $2048$ &  \\ \hline \hline  
$H_m$ &  8 & 9 & 10 & 8 & 9 & 10  & 7 & 8 \\ \hline 
{\rm Attack \cite{original} } &  $42$s & $300$s  & $7.5$m & $8.1$m & $1.1$h  &$1.96$h &$1.17$h& $2.15$h  \\ \hline  
{\rm Parallel Attack } &  $<1$s & $6.54$s  & $15.56$s & $11.44$s & $70$s  &$284$s &$70$s& $11.46$m      \\ \hline 
\end{tabular}
\caption{We used $b=n/2, Q=12,$ where $n=84,130$ and $232,$ for ${\rm{len}}(p)=600,1024$ and $2048$ bits, respectively. For each column, we executed $10$ times the attack of \cite{original} and our parallel version (algorithm \ref{3}), and we computed the average CPU time. } 
\label{Tab:results-1}
\end{table} 
{\small{
\begin{table}[h!]
{
\begin{center}
\begin{tabular}{|l||lll||lll||lll|}        \hline
 ${\rm len}(p)$   & & \ \  $600$ &  &  &$1024$ & &  \ \ \ \ & $2048$  & \\ \hline 
 $n$   & & \ \  $84$ &  &  &$130$ & &  \ \ \ \ & $232$  & \\ \hline \hline
 $H_m$            & 12       & 13       & 14         & 11       & 12     & {\bf{13}}     &   9      & 10      & {\bf 11}\\ \hline
 Parallel (time)  & $4$m     & $7$m     & $13$m   & $26$m    &  $66$m   & $142$m      & $52$m    & $107$m  & $895$m \\ \hline
 Mem. (GB)        & 1.43     & 4.93     & 7.5      & 14.28    & 21.67  & 67.41         & 26.9     & 40.61   & 127.53\\ \hline
Average Rounds    & 4.2      & 2.4      & 2.6        & 8.2      & 2.6    & 5.6           & 7.4      & 2.4     & 10\\ \hline
\end{tabular}
\caption{Extension of table \ref{Tab:results-1}. For the case $H_m=13,\ Q=12,$ and
 ${\rm{len}}(p)=1024$ we used $b=60$ instead of $b=\frac{n}{2}=65.$ 
 For the case $H_m=11$ and  ${\rm{len}}(p)=2048$ we used  $b=n/2-23=93.$
 For all other cases we used $b =\frac{n}{2}.$ The last row, Average
  Rounds, is the (average) round that eventually our algorithm terminates. Theoretically is approximately $\frac{1}{{\mathbb{P}}}$ (see Lemma \ref{lemma:step7}).}
\label{Tab:results-2}
\end{center}
}
\end{table}
}}

\begin{figure}[!htb]
%\raggedright
\includegraphics[scale=0.45]{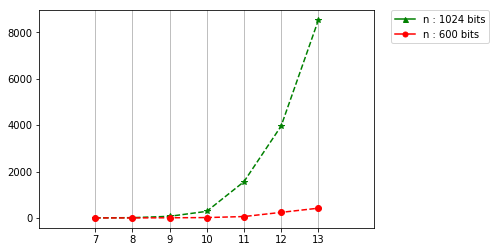}
\caption{The horizontal axis is the hamming weight $H_m$ and the vertical axis is the cpu time in seconds (for the parallel attack). Using {\tt FindFit} of Mathematica \cite{mathematica}, we computed the following approximation formulas that best fit to our data, $T_{600}(H_m) = 0.003 \cdot e^{0.71H_m},$  $T_{1024}(H_m) = 0.029 \cdot e^{0.76H_m}$ and  $T_{2048}(H_m) = 0.74\cdot e^{0.9H_m}$ (seconds).}\label{fig:1}
\end{figure}

\subsection{Carmichael Numbers}\label{section:carmichael}
 Fermat proved that if $p$ is a prime number, then $p$ divides $a^p-a$ for every integer $a.$ This is known as Fermat's Little Theorem. The question if the converse is true has negative answer. In fact in 1910 Carmichael noticed that $561$ provides such a counterexample. A Carmichael number\footnote{See also, \url{http://oeis.org/A002997}} is a positive composite integer $n$ such that $a^{n-1}\equiv 1 \pmod{n}$ for every integer $a,$ with $1<a<n$ and $\gcd(a,n)=1.$  They named after Robert Daniel Carmichael (1879-1967). Although, the Carmichael numbers between $561$ and $8911$ i.e. the first seven, initially they discovered by the Czech mathematician V. ${\rm{\breve{S}}}$imerka in 1885 \cite{simerka}. 
 In 1910 Carmichael  conjectured  that  there  is  an  infinite  number  of Carmichael numbers. This conjecture was proved in 1994 by Alford, Granville, and Pomerance \cite{annals}. Although, the problem if there  are infinitely  many Carmichael numbers with exactly $R\geq 3$ prime factors, is remained open until today.
We have the following criterion.
\begin{proposition}(Korselt, 1899, \cite{korselt})
A positive integer $n$ is Carmichael if and only if is composite, square-free and for every prime $p$ with $p|n,$ we get $p-1|n-1.$
\end{proposition}
For a simple elegant proof see \cite{chernick}.
We define the following function.
$$\lambda(2^{a})=\phi(2^{a})\ \text{if}\ a=0,1,2$$
and 
$$\lambda(2^{a})=\frac{1}{2}\phi(2^a),\text{if}\ a>2.$$
If $p_i$ is odd prime and $a_i$ positive integer
$$\lambda(p_i^{a_i})=\phi(p_i^{a_i}).$$
If $n=p_1^{a_1} p_2^{a_2}\cdots p_k^{a_k}$ then
$$ \lambda(n) = \operatorname{lcm}\big(\lambda(p_1^{a_1}),\,\lambda(p_2^{a_2}),\,\ldots,\,\lambda(p_k^{a_k}) \big).$$
Korselt's criterion can be written as
\begin{proposition}(Carmichael, 1912, \cite{Robert_Carmichael}) $n$ is Carmichael if and only if is composite and 
$n\equiv 1\pmod{\lambda(n)}.$
\end{proposition}
Using the previous, we can prove that a Carmichael number is odd and have at least three prime factors.
Furthermore, we can calculate some Carmichael numbers (the first 16): 561, 1105, 1729, 2465, 2821, 6601, 8911, 10585, 15841, 29341, 41041, 46657, 52633, 62745, 63973, and 75361. 
In \cite{carmichael} they used an idea of Erd\H{o}s \cite{erdos} to find Carmichael numbers with many prime factors. In 1996, Loh  and Niebuhr \cite{loh} provided a Carmichael with $1,101,518$ prime factors using Erd\H{o}s heuristic algorithm (Algorithm \ref{1}). Also, an analysis and some refinements of \cite{loh} and an extension to other pseudoprimes was provided by Guillaume and Morain in 1996 in \cite{extension_of_loh}. Further, in the same paper the authors provided a Carmichael number having 5104 prime factors. 
In 2014 \cite[Table 1]{carmichael} the authors provided two large Carmichael numbers with many prime factors. The first one with $1,021,449,117$ prime factors and $d = 25,564,327,388$ decimal digits and the second with $10,333,229,505$ prime factors, with $d=29,548,676,178$ decimal digits. 

Also, in 1975 J. Swift \cite{swift} generated all the Carmichael numbers below $10^{9}.$ 
In 1979, Yorinaga \cite{yorinaga} provided a table for Carmichael numbers up to $10^{10}$ using the  method of Chernick (this method allow us to construct a Carmichael  number having already one, see \cite[Theorem 2.2]{extension_of_loh}). In 1980, Pomerance, Selfridge and Wagstaff \cite{pom}  generated Carmichael numbers up to $25\cdot 10^{10}.$ In 1988 Keller \cite{keller} calculated the Carmichael numbers up to $10^{13}.$ In 1990, Jaeschke \cite{Jaeschke} provided tables for Carmichael numbers up to $10^{12}.$ 
Pinch provided a table for all Carmichael numbers up to $10^{18}$ (\cite{pinch3}). Also the same author in 2006 \cite{pinch4} computed all Carmichael numbers up to $10^{20}$ and in 2007 \cite{pinch5} a table up to $10^{21}.$ Furthermore, he found $20138200$ Carmichael numbers up to $10^{21}$ and all of them have at most $12$ prime factors. 

For an illustration of our algorithm we also generated some tables for Carmichael numbers having many prime factors\footnote{see, \url{https://github.com/drazioti/Carmichael} }. For instance we produced Carmichael numbers up to $250$ prime factors. Each instance was generated in some seconds.  Also Carmichael numbers with $11725$ and $19589$ prime factors were generated in some hours with our algorithm, in a small home PC (I3/16Gbyte) using a C++/gmp implementation.

The following method is based on Erd\H{o}s idea \cite{erdos}. It was used in \cite{carmichael,loh} to produce Carmichael numbers having large number of prime factors. 
\\\\
{\bf Algorithm  \refstepcounter{myalgorithm}\label{4} \arabic{myalgorithm} : Generation of Carmichael Numbers}\\
	{\texttt{INPUT: A positive integer $r$ and a vector ${\bf H}=(h_1,h_2,...,h_r)\in {\bf Z}^{r},$ with $h_1\geq h_2\geq\cdots \geq h_r\geq 1.$ Also we consider two positive integers $\ell$ and $b$ which correspond to the local hamming and the bound, respectively.}}\\
	 {\texttt{OUTPUT: A Carmichael number or Fail}}\\\\
	 \texttt{1:}   {\texttt{$Q\leftarrow \{q_1,...,q_r\}$ the $r-$first prime numbers}}\\
     \texttt{2:}   $\Lambda\leftarrow q_1^{h_1}\cdots q_r^{h_r}$ \\
    \texttt{3:}    ${\mathcal{P}}\leftarrow \{p: p\ \text{prime}\ p-1|\Lambda,\ p\not{|}\Lambda\}$\\
	\texttt{4:}    $S\leftarrow{\rm BA\_MSPP}_{\Lambda}({\mathcal{P}},1;b,\ell) $%${\mathcal{•}l{S}}\leftarrow MSPP_{\Lambda}({\mathcal{P}},1)$
	\\
    \texttt{5:}    {\bf If} $|S|\geq 2$ {\bf return} $\prod_{p\in {S}}p$\\
    \texttt{6:}    $B\leftarrow \prod_{p\in\mathcal{P}}p$\\
   \texttt{7:}    ${{T}}\leftarrow{\rm BA\_MSPP}_{\Lambda}({\mathcal{P}},B;b,\ell) $%${\mathcal{T}}\leftarrow MSPP_{\Lambda}({\mathcal{P}},b)$
   \\
    \texttt{8:}    {\bf If} $|{T}|\geq 2$ {\bf return} $B/\prod_{p\in T}p$\\
     \texttt{9:}   {\bf else return}} Fail\\\\
  In line 8, we return the number $\displaystyle \prod_{p\in \mathcal{P}-{T}}p.$\\    
{\bf Correctness}.\\
It is enough to prove that the numbers returned in steps 5 and 8 are Carmichael. Set $n=\prod_{p\in{{S}}}p.$ 
We shall prove it for step 5.  The set $S$ contains all the primes of $\mathcal{P}$ such that their product is equivalent to $1\pmod{\Lambda}.$ Since $|{S}|\geq 2,$ $n$ is composite and also is squarefree. Say a prime $p$ is such that $p|n.$ Since $n\equiv 1\pmod{\Lambda}$ i.e. $\Lambda|n-1$ we get $p-1|n-1.$ Indeed, this is immediate since $p-1|\Lambda.$ From Korselt's criterion we get that $n$ is Carmichael. Similar for the step 8.\\\\
 We have set $\ell = local\_hamming.$ In case of success, the output of the algorithm is a Carmichael number with $\ell$ or $|{\mathcal{P}}|-\ell$ prime factors. In fact, if we want to calculate a Carmichael number with many prime factors, we can ignore the lines 4 and 5 and consider a large set ${\mathcal{P}}.$
 An estimation for $|\mathcal{P}|$ was given in \cite[formula 4]{loh},
$$|\mathcal{P}|\approx g(\Lambda)\prod_{j=1}^{r}\Big( h_j+\frac{q_j-2}{q_j-1}\Big),\ \text{where}\ g(\Lambda)=\frac{\Lambda}{\phi(\Lambda)\ln{\sqrt{2\Lambda}}}.$$

In lines 1 and 2 we initialize the algorithm. Since in practice $r$ is not large enough, both these steps are very efficient.

In line 3 we calculate the set  $\mathcal{P}.$ 
One way to construct this set is the following. Say $d|\Lambda.$ If $d+1$ is prime with $d+1\not \in Q$ then $d\in {\mathcal{P}}.$ To find the divisors of $\Lambda$ having their prime divisors is a simple combinatorial problem. We can implement this without using much memory. Even better, we can use \cite[Section 8]{carmichael} where they keep only the exponents  of the divisors of $\Lambda.$ Since the set ${\mathcal{P}}$ contains integers of the form $2^{a_1}3^{a_2}\cdots p_r^{a_r}+1$ with $0\leq a_i\leq h_i,$ instead of storing $2^{a_1}3^{a_2}\cdots p_r^{a_r}+1$ we can store $(a_1,...,a_r).$ Overall $8r|\mathcal{P}|$ bits or $r|\mathcal{P}|$ bytes. So the set ${\mathcal{P}}$ is {\it nice}, since the set $B$ in formula (\ref{formula:memory}) needs $O(|{\mathcal{P}}|\log_{2}(|{\mathcal{P}}|))$ bits for storage.

 In line 4 (and 7), we use algorithm \ref{2} with $b = bound$ and $\ell = local\_hamming$ according to the user choice. We can apply  ${\rm BA\_MSPP}$ with the parameters $Q$ and 
${\rm{iter}},$ ${\rm BA\_MSPP}_{\Lambda}({\mathcal{P}},c;b,\ell,Q,{\rm iter}).$ In \cite{loh} they picked ${T}$ randomly from $\mathcal{P}.$ 
\begin{remark}
In \cite{ carmichael} they used another algorithm inspired  by the quantum algorithm of Kuperberg and they exploit the distribution of the primes in the set $\mathcal{P}$ (which is not uniform).
\end{remark}
\begin{remark}
When $|\mathcal{P}|$ is large enough then using $B=1$ as target number we can easily find a Carmichael number with small number of prime factors (by using small local Hamming weight). If we use $B>1$ as  in line 5 we get a Carmichael number with many prime factors. As we remarked previous the number of prime factors of the Carmichael number is either $\ell$ or $|{\mathcal{P}}|-\ell.$ One advantage of the algorithm is that we can search for Carmichael numbers near $|{\mathcal{P}}|-r.$ This can be done by considering $\ell={{local\_hamming}}$ close to $r.$ In this way we quickly generated Carmichael numbers up to $250$ prime factors in a small PC.
\end{remark}

\section{Conclusions}
In the present work we considered a parallel algorithm to attack
the modular version of product subset problem. This is  a NP-complete problem which have many applications in computer science and mathematics. Here we provide two applications, one in number theory and the other to cryptography.

%The attack is based on the meet in the middle method.
First we applied our algorithm (providing a C++ implementation)
to the the problem of searching Carmichael numbers. We managed to find one with 19589 factors in a small PC in 3 hours. %e believe that this can further be improved by using a better computer and a suitable parallel version.

For the Naccache-Stern knapsack cryptosystem we updated and extended previous experimental cryptanalytic results provided in \cite{original}. The new bounds for $H_m$ concern messages having Hamming weight $\leq 11$ or $\geq 223,$ for $n=232.$ This is proved by providing experiments. But, our attack is feasible for Hamming weight $\leq 15$ or $\geq 219.$ The NSK cryptosystem system could resist to this attack, if we consider  Hamming weights in the real interval $[17,217].$
\\\\
{\bf Acknowledgments.}
The authors are grateful to High Performance Computing Infrastructure and Resources (HPC) of the Aristotle's University of Thessaloniki (AUTH, Greece), for providing access to their computing facilities and their technical support.

\end{document}